\documentclass{article}
\usepackage[utf8]{inputenc}
\usepackage{amsmath, amssymb, enumerate, multicol, chngcntr}
\usepackage[exercise]{amsthm}
\usepackage{xcolor}
\usepackage{pdfpages}
\usepackage{comment}
\usepackage{caption}
\usepackage{bm}
\usepackage{mathtools}

\theoremstyle{plain}
\newtheorem{thm}{Theorem}[section]
\newtheorem{cor}[thm]{Corollary}
\newtheorem{prop}[thm]{Proposition}
\newtheorem{lemm}[thm]{Lemma}
\newtheorem{rem}{Remark}

\counterwithin{equation}{section}

\usepackage{fancyhdr}
\usepackage{setspace}

\usepackage[english]{babel}
\usepackage[euler]{textgreek}
\usepackage{stmaryrd}

\usepackage{tikz}
\usetikzlibrary{shapes,arrows}
\usepackage{tikz-cd}
\usepackage{amssymb}
\usepackage{graphicx}
\graphicspath{{Images/}}
\usepackage[font=small,labelfont=bf]{caption}
\usepackage{pgf,tikz,pgfplots}
\pgfplotsset{compat=1.15}
\usepackage{mathrsfs}
\usepackage{diagbox}
\usepackage{float}
\usepackage{xcolor}
\usepackage{hyperref}
\usepackage{ragged2e}
\usepackage{enumerate}
\usepackage{upgreek}
\usepackage{multicol}

\theoremstyle{remark}

\theoremstyle{plain}

\newtheoremstyle{note}
  {3pt}
  {3pt}
  {}
  {}
  {\itshape}
  {:}
  {.5em}
  {}

\newtheoremstyle{citing}
  {3pt}
  {3pt}
  {\itshape}
  {}
  {\bfseries}
  {.}
  {.5em}
  {\thmnote{#3}}

\theoremstyle{citing}

\newtheoremstyle{break}
  {9pt}
  {9pt}
  {\itshape}
  {}
  {\bfseries}
  {.}
  {\newline}
  {}

\let\lvert=|\let\rvert=|

\title{Enumeration of solvable cube-free groups and counting certain types of split extensions}
\author{Prashun Kumar \footnote{Corresponding author, Dr. B. R. Ambedkar University Delhi, Delhi 110006, India; \ E-mails: prashunkumar.19@stu.aud.ac.in,  prashun07kumar@gmail.com.}
\ and \ Geetha Venkataraman \footnote{ Dr. B. R. Ambedkar University Delhi, Delhi 110006, India; \ E-mails: geetha@aud.ac.in,  geevenkat@gmail.com.}}

\begin{document}
\fontfamily{cmr}\selectfont

\maketitle


\bigskip
\noindent
{\small{\bf ABSTRACT:}}
A group is said to be cube-free if its order is not divisible by the cube of any prime. Let $f_{cf,sol}(n)$ denote the isomorphism classes of solvable cube-free groups of order $n$. We find asymptotic bounds for $f_{cf,sol}(n)$ in this paper. Let $p$ be a prime and let $q = p^k$ for some positive integer $k$. We also give a formula for the number of conjugacy classes of the subgroups that are maximal amongst non-abelian solvable cube-free $p'$-subgroups of ${\rm GL}(2,q)$. Further, we find the exact number of split extensions of $P$ by $Q$ up to isomorphism of a given order where $P \in \{{\mathbb Z}_p \times {\mathbb Z}_p, {\mathbb Z}_{p^{\alpha}}\}$, $p$ is a prime, $\alpha$ is a positive integer and $Q$ is a cube-free abelian group of odd order such that $p \nmid |Q|$. 

\medskip
\noindent
{\small{\bf Keywords}{:}}
cube-free groups, conjugacy classes of subgroups, enumeration of finite groups, general linear groups, solvable groups.

\medskip
\noindent
{\small{\bf Mathematics Subject Classification-MSC2020}{:} }
20E28, 20E34, 20E45, 20F99.

\baselineskip=\normalbaselineskip
\section{Introduction}

 Enumeration of finite groups of a given order, up to isomorphism is a well known problem in group theory and has an old history.

 In 1893, Otto H$\ddot{\text{o}}$lder described groups of order $p^3$ and $p^4$. Soon after, he arrived at a formula for the number of groups of order $n$ when $n$ is square-free, that is, the square of no prime divides $n$. Let $f(n)$ denote the number of groups of order $n$ up to isomorphism. For square-free $n$,  H$\ddot{\text{o}}$lder's formula is given below.

 $$f(n) = \sum_{m \mid n} \prod_p \frac{p^{c(p)} -1}{p-1}$$

  where $p$ is a prime divisor of $n/m$ and $c(p)$ is the number of prime divisors $q$ of $m$ that satisfy $q \equiv 1\mod{p}$ (see \cite{H1893}, \cite{H1895}).

\medskip
In 1960, Graham Higman \cite{H1960} showed that
  
          $$f(p^m) \geq p^{\frac{2}{27}m^3 - O(m^2)},$$
          
  and a few years later in 1965 Charles Sims \cite{S1965}, proved that 
   
          $$f(p^m) \leq p^{\frac{2}{27}m^3 + O(m^{8/3})}.$$

In his paper Sims conjectures that the error term in the estimate could be $O(m^2)$ rather than $O(m^{8/3})$. An unpublished work of Mike Newman and Craig Seely brings down the error term to $O(m^{5/2})$. 
  
\medskip
{L{\'a}szl{\'o}} Pyber \cite{P1993} proved in 1991, (published in 1993), that
   
        $$f(n) \leq n^{\frac{2}{27}\mu(n)^2 + O(\mu(n)^{5/3})}$$
        
  where $\mu(n)$ is the highest power to which any prime divides $n$.
  
  \medskip
  A finite group $G$ is called an $A$-group if all the Sylow subgroups of $G$ are abelian. Clearly every cube-free group is an $A$-group.

  \medskip
  {L{\'a}szl{\'o}} Pyber also found an upper bound for the number of $A$-groups up to isomorphism (see \cite{P1993}). The upper bound for the number of solvable $A$-groups was further improved by Geetha Venkataraman in 1997 (see \cite{GV1997}). Therefore we also have a bound for the number of cube-free groups of a given order.

  \medskip
  The structure of cube-free groups was studied by Heiko Dietrich and Bettina Eick. They also presented an algorithm to construct cube-free groups up to isomorphism of a given order using ${\rm GAP}$ (see \cite{DE2005}). Later, S. Qiao and C. H. Li gave a more explicit structure of cube-free groups (see \cite{QL2011}). 
  
 \medskip
 The objective of this paper is to find asymptotic bounds for the number of cube-free groups of a given order $n$ up to isomorphism. Our upper bound is an improvement over the previous upper bound available through enumerating $A$-groups. Let $p$ be a prime and let $q = p^k$ for some positive integer $k$. In the process of finding upper bounds we give a formula for the number of conjugacy classes of the subgroups that are maximal amongst non-abelian solvable cube-free $p'$-subgroups of ${\rm GL}(2,q)$. Further, in the process of finding the lower bound for the number of cube-free groups of a given order up to isomorphism we will find closed formulae for the number of split extensions of $P$ by $Q$ where $P \in \{{\mathbb Z}_p \times {\mathbb Z}_p, {\mathbb Z}_{p^{\alpha}} \}$, $p$ is a prime, $\alpha$ is a positive integer and $Q$ is a cube-free abelian group of odd order such that $p \nmid |Q|$.   

 \medskip
 Let $n$ be a positive integer. We define $\omega(n)$ as the number of prime divisors of $n$ and $\tau(n)$ to be the number of divisors of $n$. 

 \medskip
 Let $G$ be a finite group and let $H \leq G$. We denote $[H]_G$ to be the conjugacy class of $H$ in $G$. Let $p$ be a prime and let $p \mid |G|$. Then we denote $O_{p'}(G)$ to be the largest normal $p'$-subgroup of $G$.

 \medskip
  Let $N$ and $Q$ be groups and let $\theta: Q \longrightarrow {\rm Aut}(N)$ be a homomorphism. A split extension $G$ of $N$ by $Q$ realises $\theta$ if, for all $x \in Q$ and $a \in N$
       $$\theta_x(a) = xax^{-1}$$
and we write $G = N \rtimes_{\theta} Q$.

\medskip
 Let $N$ and $Q$ be finite abelian groups with $\gcd(|N|,|Q|)=1$ and let ${\cal X}_{N,Q} = {\rm Aut}(N) \times {\rm Aut}(Q)$. Let ${\Gamma}_{N,Q}$ be the set of homomorphisms $\theta: Q \longrightarrow {\rm Aut}(N)$. There is an action of ${\cal X}_{N,Q}$ on ${\Gamma}_{N,Q}$ defined as follows. Let $(\kappa,\lambda) \in {\cal X}_{N,Q}$ where $\kappa \in {\rm Aut}(N)$ and $\lambda \in {\rm Aut}(Q)$. Let $\theta \in \Gamma_{N,Q}$. Then we define $(\kappa, \lambda)\theta = \theta'$, where $\theta'$ is defined by 

 \begin{equation*}\label{eq: action}
     \theta'_{x}(a) = \kappa\theta_{\lambda^{-1}(x)}\kappa^{-1}(a)\tag{$*$}
 \end{equation*}
 
 for all $a \in N$ and $x \in Q$.

\medskip
 Throughout the paper, $p$ is a prime, $q$ is a power of $p$ and $\mathbb{F}_q$ is the finite field of order $q$. Let $D(2,q)$, denote the subgroup of diagonal matrices of $\mbox{GL}(2,q)$. Any $A \in D(2,q)$ with diagonal entries $\lambda_1$ and $\lambda_2$ will be represented as ${\rm diag}(\lambda_1,\lambda_2)$. Let $M(2,q) = D(2,q) \rtimes \langle a \rangle$ be the subgroup of monomial matrices in ${\rm GL}(2,q)$, where $a = 
 \begin{pmatrix}
    0 & 1\\
    1 & 0
\end{pmatrix}$. Let $N(2,q)$ be  the normaliser of $S(2,q)$ where $S(2,q) \cong \mathbb{Z}_{q^2-1}$ is a Singer cycle.

\medskip
Now we shall state the main results of our paper.

\begin{thm}{\label{Enumeration_of_odd_cube_free_groups}}
Let $n$ be an odd positive cube-free integer. If $f_{cf,odd}(n)$ is the number of groups of order $n$ up to isomorphism then

\begin{equation*}
   f_{cf,odd}(n) \leq  2^{-\omega(n)\mu(n)}n^{2\mu(n) - 1/4}.
\end{equation*}    
\end{thm}

\begin{thm}\label{Enumeration_of_cube_free_groups}
 Let $n$ be cube-free. If $f_{cf,sol}(n)$ is the number of solvable groups of order $n$ up to isomorphism then
 
\begin{equation*}
    f_{cf,sol}(n) \leq {n^{2\mu(n)}}.
\end{equation*}

\end{thm}

The result below will give the formulae for the number of split extensions of $N$ by a cube-free abelian odd order group $Q$ where $p \nmid |Q|$ and $N \in \{{\mathbb Z}_p \times {\mathbb Z}_p, {\mathbb Z}_{p^{\alpha}}\}$ for some positive integer $\alpha$. 

\medskip
Let $Q$ be a cube-free abelian group. It should be noted that any two subgroups of $Q$ of the same order are isomorphic and any such $Q \cong {\mathbb Z}_{cd} \times {\mathbb Z}_d$ where $c$ is cube-free, $d$ is square-free and $\gcd(c,d) = 1$.

\medskip
Let $P$ and $Q$ be finite abelian groups where $P \cong \langle a_1 \rangle \times \cdots \times \langle a_k \rangle$ such that $|a_i| = m_i$ and $m_{i+1} \mid m_{i}$ for all $i$ and $Q \cong \langle b_1 \rangle \times \cdots \times \langle b_r \rangle$ such that $|b_i| = n_i$ and $n_{i+1} \mid n_{i}$ for all $i$. Then we denote the set of split extensions of $P$ by $Q$ by ${\cal S}_{(m_1, \ldots, m_k),(n_1, \ldots , n_r)}$. 

\medskip
Let $H(t)$ be a group isomorphic to a reducible $p'$-subgroup of ${\rm GL}(2,p)$ of order $t$ and let $N_{red}(H(t))$ denote the number of conjugacy classes of reducible subgroups of ${\rm GL}(2,p)$ isomorphic to $H(t)$. Similarly if $K(v)$ is a group isomorphic to an irreducible $p'$-subgroup of ${\rm GL}(2,p)$ of order $v$, then $N_{irr}(K(v))$ denotes the number of conjugacy classes of irreducible subgroups of ${\rm GL}(2,p)$ that are isomorphic to $K(v)$. 

\begin{thm}{\label{isomorphism_1}}
      Let $c$ and $d$ be odd positive integers such that $p \nmid cd$, $c$ is cube-free, $d$ is square-free and ${\rm gcd}(c,d) = 1$. Let $l = \gcd(cd,p-1)$, $m = \gcd(d,p-1)$ and $s = \gcd(cd, p+1)$.

     \begin{enumerate}[$(\rm i)$]
         \item Let $f_{\cal S}(p^2cd^2)$ denote the number of isomorphism classes of the groups in $ {\cal S} = {\cal S}_{(p,p),(cd,d)}$. Then 

  $$f_{\cal S}(p^2cd^2) = \displaystyle\sum_{t \mid lm} {N_{red}}(H(t)) + \Omega(s,l)$$
where 

 $\Omega(s,l) =
\begin{cases}
   0 & \ \mbox{if } \ s = 1, \\
   \tau(l)(\tau(s)-1) &  \ otherwise.
\end{cases}$

\item Let $f_{\cal S}(p^{\alpha}cd^2)$ denote the number of isomorphism classes of groups in ${\cal S} = {\cal S}_{p^{\alpha},(cd,d)}$. Then 

  $$f_{\cal S}(p^{\alpha}cd^2) = \tau(l).$$

\end{enumerate}

\end{thm}

\begin{rem}
    Let $H(t)$ be a cube-free reducible $p'$-subgroup of ${\rm GL}(2,q)$ of order $t$. Then ${N_{red}}(H(t))$ can be evaluated by using \cite[Theorem 1.2]{PKGV2024}.
\end{rem}

The result below gives us a lower bound for the number of cube-free groups up to isomorphism of a given odd order. The result will be stated using the classification of cube-free groups by S. Qiao and C. H. Li (see \cite{QL2011}). 

\medskip
Before stating our result we first note that if $G$ is a group of cube-free odd order, then by \cite[Theorem 3.9]{QL2011}  we have $G \cong ({\mathbb Z}_{ab} \times {{\mathbb Z}_{b}}) \rtimes ({\mathbb Z}_{cd} \times {{\mathbb Z}_d})$ where $a,b,c,d$ are suitable integers such that $\gcd(a,b)=1=\gcd(c,d)$, $ac$ is cube-free, $bd$ is square-free, prime divisors of $ab$ are not less than prime divisors of $cd$. In order to construct cube-free groups using \cite[Theorem 3.9]{QL2011} of S. Qiao and C. H. Li we will impose an additional condition on $({\mathbb Z}_{ab} \times {{\mathbb Z}_{b}}) \rtimes ({\mathbb Z}_{cd} \times {{\mathbb Z}_d})$ that $\gcd(a,c) = 1$.  The lower bound will be given in terms of the formulae we have obtained in the Theorem \ref{isomorphism_1}.   

\begin{cor}{\label{lowr_bound}}
   Let $a,b,c,d$ be mutually co-prime odd positive integers such that $b$ and $d$ are square-free, $a$ and $c$ are cube-free. Let $\{p_1, \ldots , p_k \}$ be the collection of all the prime divisors of $ab$ and let $P_i$ be the Sylow $p_i$-subgroup of ${\mathbb Z}_{ab} \times {\mathbb Z}_{b}$ of order ${p_i}^{\alpha_i}$ where $\alpha_i \in \{1,2\}$. Let ${\cal S}_{P_i,(cd,d)}$ be the class of split extensions of $P_i$ by ${\mathbb Z}_{cd} \times {\mathbb Z}_d$. Let $f_{\cal S}(a{b^2}c{d^2})$ denote the number of isomorphism classes of groups in ${\cal S}={\cal S}_{(ab,b),(cd,d)}$. Let $f_{{\cal S}_{{P_i,(cd,d)}}}({p_i}^{\alpha_i}c{d^2})$ denote the number of isomorphism classes of split extensions of $P_i$ by ${\mathbb Z}_{cd} \times {\mathbb Z}_d$. Then
   
    $$f_{\cal S}(a{b^2}c{d^2}) \geq \displaystyle\prod_{i=1}^{k} f_{{\cal S}_{{P_i,(cd,d)}}}({p_i}^{\alpha_i}c{d^2}).$$
\end{cor}

The paper is organised as follows. In Section \ref{cnjgcy_clss_of_mxml_sbgrps}  we first find an upper bound for the number of conjugacy classes of abelian subgroups that are maximal amongst solvable cube-free $p'$-subgroups of ${\rm GL}(2,q)$. Then we find a closed formula for the number of conjugacy classes of non-abelian subgroups that are maximal amongst solvable cube-free $p'$-subgroups of ${\rm GL}(2,q)$. In Section \ref{uppr_bound_of_cbe_free_grps} we first give a formula for the number of elementary abelian subgroups of a given order of a finite abelian group of exponent $p^2$ and also a formula for the number of cyclic subgroups of order $p^2$. Then we prove Theorem \ref{Enumeration_of_odd_cube_free_groups} and Theorem \ref{Enumeration_of_cube_free_groups}. Let $N$ and $Q$ be finite abelian groups with $\gcd(|N|,|Q|)=1$ and let ${\cal X}_{N,Q} = {\rm Aut}(N) \times {\rm Aut}(Q)$. Let ${\Gamma}_{N,Q}$ be the set of homomorphisms $\theta: Q \longrightarrow {\rm Aut}(N)$. Let ${\cal X}_{N,Q}$ acts on ${\Gamma}_{N,Q}$ by the equation (\ref{eq: action}). In Section \ref{lowr_bound_of_cbe_free_grps} we first prove the equivalence between the ${\cal X}_{N,Q}$ orbits of ${\Gamma}_{N,Q}$ and the isomorphism classes of $P$ by $Q$. Further suppose that Sylow subgroups of $Q$ are either elementary abelian or cyclic then we will prove that there is an equivalence between the isomorphism classes of split extensions of $N$ by $Q$ and the conjugacy classes of images of orbit representatives of elements of ${\Gamma}_{N,Q}$ under the action of ${\cal X}_{N,Q}$. Finally we prove Theorem \ref{isomorphism_1} and Corollary \ref{lowr_bound}.

\section{Conjugacy classes of maximal solvable cube-free subgroups of \texorpdfstring{${\rm GL}(2,q)$}{GL(2,q)}}\label{cnjgcy_clss_of_mxml_sbgrps}

In this section we will find an upper bound on the number of conjugacy class of subgroups that are maximal amongst cube-free solvable $p'$-subgroups of ${\rm GL}(2,q)$. These results will later be used to find an upper bound for the number of cube-free groups of a given order up to isomorphism.

\medskip
We start with the case of counting the number of conjugacy classes of subgroups that are maximal amongst cube-free abelian $p'$-subgroups of ${\rm GL}(2,q)$. As this number differs for even and odd order subgroups we deal with the subgroups of even and odd orders separately. Proofs of both odd and even cases are essentially the same, so we only provide the proof for the odd case.

\begin{lemm}\label{No_of_cnjgcy_clsses_of_ablian_grps}
Let $p$ be an odd prime and let $q = p^k$ for some positive integer $k$. Then
\begin{enumerate}[{$(\rm i)$}]
     \item the number of conjugacy classes of odd order subgroups that are maximal amongst abelian cube-free $p'$-subgroups of ${\rm GL}(2,q)$ is at most $\frac{1}{2}(q-1)^{3/2} + 1$.
     \item the number of conjugacy classes of even order subgroups that are maximal amongst abelian cube-free $p'$-subgroups of ${\rm GL}(2,q)$ is at most $\frac{3}{4}(q-1)^{3/2} + 1$.
\end{enumerate}
\end{lemm}

\begin{proof} Let $M \leq {\rm GL}(2,q)$ be maximal amongst abelian cube-free $p'$-subgroups of odd order. Then either $M$ is reducible or $M$ is irreducible cyclic. Let $M$ be reducible. Then by \cite[Lemma 1.1]{PKGV2024} $M$ is conjugated to a subgroup of $D(2,q)$. Assume $M \leq D(2,q)$. Let $\{p_1, \ldots , p_k\}$ be the set all odd prime divisors of $q-1$. Due to maximality of $M$ we must have $|M| = {p_1}^2 \ldots {p_k}^2 = m$. Let ${\cal O}$ be the complete and irredundant set of representatives of isomorphism classes of subgroups of $D(2,q)$ of order $m$. Let $H \in {\cal O}$ and let ${\cal I}(H) = \{i \mid$ the Sylow $p_i$-subgroup of $H$ is cyclic$\}$.  By \cite[Theorem  1.2]{PKGV2024}  the number of conjugacy classes of reducible $p'$-subgroups isomorphic to $H$ in ${\rm GL}(2,q)$ is at most $ \frac{1}{2}\prod_{i \in {\cal I}(H)}2{p_i}^2$

  Let ${\cal P} = \{i \mid {p_i}^2 $ divides $q-1\}$. Let $N$ be the number of conjugacy classes of subgroups that are maximal amongst abelian cube-free reducible $p'$-subgroups of ${\rm GL}(2,q)$. Then we have

  \begin{align*}
      N & \leq \frac{1}{2}(q-1) \sum_{H \in {\cal O}} 2^{|{\cal I}(H)|}\\
        & = \frac{1}{2}(q-1) \sum_{d = 0}^{t}\sum_{H \in {\cal O}_d} 2^{|{\cal I}(H)|}\\
        & = \frac{1}{2}(q-1)\sum_{d = 0}^{t}|{\cal O}_d|2^d
  \end{align*}

  where ${\cal O}_d = \{H \in O \mid |{\cal I}(H)| = d\}$ and $t = |{\cal P}|$. Clearly the $|{\cal O}_d|$ is the number of subsets of ${\cal P}$ of order $d$. Therefore 
  \begin{align*}
      N & \leq \frac{1}{2}(q-1)3^t\\
        & \leq \frac{1}{2}(q-1)^{3/2}.
  \end{align*}
 If $M$ is a maximal abelian cube-free irreducible $p'$-subgroup of ${\rm GL}(2,q)$, then it is conjugated to a subgroup of the Singer cycle and so there will be at most one choice for the conjugacy classes of such subgroups. Thus the number of conjugacy classes of subgroups that are maximal amongst abelian odd cube-free $p'$-subgroups  of ${\rm GL}(2,q)$ is at most $\frac{1}{2}(q-1)^{3/2} + 1$.  
\end{proof}

\begin{cor}\label{No_of_odd_cnjgcy_clsses}
   Let $p$ be an odd prime. Then the number of conjugacy classes of odd order subgroups that are maximal amongst cube-free $p'$-subgroups  of ${\rm GL}(2,p)$ is at most $\frac{1}{2}p^{3/2} - 1$.
\end{cor}

\begin{proof}
    Let $N$ be the number of conjugacy classes of odd order subgroups that are maximal amongst solvable cube-free $p'$-subgroups of ${\rm GL}(2,p)$. Let $M$ be an odd order subgroup that is maximal amongst solvable cube-free $p'$-subgroups of ${\rm GL}(2,p)$. Then by \cite[Lemma 1.1]{PKGV2024} $M$ is abelian. Now if $p = 3$, then the only maximal odd order cube-free $p'$-subgroup of ${\rm GL}(2,3)$ is the trivial subgroup. Hence $N = 1 < \frac{1}{2}3^{3/2} - 1$. 
    
    If $p \in \{5,7\}$, then a maximal reducible cube-free odd order $p'$-subgroup of ${\rm GL}(2,p)$ would be trivial or conjugate to the unique subgroup of $D(2,p)$ isomorphic to ${\mathbb Z}_l \times {\mathbb Z}_l$ for some $l \mid p-1$. Thus the number of conjugacy classes of subgroups that are maximal amongst cube-free reducible odd order $p'$-subgroups is $1$. If $M$ is a maximal abelian cube-free irreducible $p'$-subgroup of ${\rm GL}(2,p)$, then it is conjugated to a subgroup of the Singer cycle and so there will be at most one choice for the conjugacy classes of such subgroups. So $ N = 2 < \frac{1}{2}p^{3/2} - 1$.
    
    If $p \geq 11$, then by Lemma \ref{No_of_cnjgcy_clsses_of_ablian_grps}, $N \leq \frac{1}{2}(p-1)^{3/2} + 1$ and using standard calculus it can be shown that $\frac{1}{2}(p-1)^{3/2} + 1 < \frac{1}{2}p^{3/2} - 1$ for $p \geq 11$. 
\end{proof}

Now we deal with the non-abelian cube-free solvable $p'$-subgroups of ${\rm GL}(2,q)$. In this situation, we in fact give the precise number for the number of conjugacy classes of subgroups that are maximal amongst non-abelian solvable cube-free $p'$-subgroups of ${\rm GL}(2,q)$.

\medskip
Note that a non-abelian cube-free solvable $p'$-subgroup is irreducible, as otherwise by \cite[Lemma 1.1]{PKGV2024} it would be conjugate to a subgroup of $D(2,q)$. 

\begin{prop}\label{No_of_cnjgcy_clsses_of_non_ablian_grps}
      Let ${\cal S}$ be the set of all odd primes that divide $q-1$ and let ${\cal P} = \{r \in {\cal S} \mid r^2 \mid q-1\}$. Then the number of conjugacy classes of subgroups that are maximal amongst non-abelian solvable cube-free $p'$-subgroups  of ${\rm GL}(2,q)$ is  $$
     \begin{cases}
         2(3^t + 1)   & if \quad {\cal S} \not= {\cal P},\\
         2\times 3^t  & if \quad {\cal S} = {\cal P}
     \end{cases}$$
     where $t =  |{\cal P}|.$
\end{prop}

\begin{proof}
     Suppose that $M$ is a non-abelian imprimitive maximal solvable cube-free $p'$-subgroup of ${\rm GL}(2,q)$, then by \cite[Lemma 1.1]{PKGV2024} $M$ is conjugated to a subgroup of $M(2,q)$.  Assume that $M \leq M(2,q)$. Then $M \cong H \rtimes P$ where $H \leq D(2,q)$ and $P$ is the Sylow  $2$-subgroup of $M$. Now we show that $|M| = 4{p_1}^2\ldots{p_s}^2$ where ${\cal S} = \{p_1,\ldots,p_s\}$. 
    
    \medskip
    Suppose that $|P| = 2$, then we can choose a scalar matrix of order $2$ and form a subgroup $L = M \times {\mathbb Z}_2$ which is a solvable cube-free, imprimitive $p'$-subgroup larger than $M$. Hence $|P| = 4$. Suppose that $r \in {\cal S}$ and $r \nmid |M|$. Let $R \leq D(2,q)$ be a subgroup isomorphic to ${\mathbb Z}_{r} \times {\mathbb Z}_{r}$. Then $R$ is normal in $M(2,q)$. Thus $L = MR$ is a solvable cube-free, imprimitive $p'$-subgroup which properly contains $M$. This also shows that if $r \in {\cal S}$, then $r^2 \mid |M|$. 
    
    \medskip
    Now by \cite[Theorem 1.3]{PKGV2024} any two isomorphic imprimitive subgroups of ${\rm GL}(2,q)$ are conjugate. Now we find the number isomorphism classes of non-abelian solvable cube-free imprimitve $p'$-subgroups of ${\rm GL}(2,q)$ of order $4m$ where $m = {p_1}^2 \ldots {p_s}^2$. By \cite[Lemma 1.1]{PKGV2024} this number is equal to the number of isomorphism classes of non-abelian cube-free imprimitive subgroups of $M(2,q)$ of order $4m$. By \cite[Lemma 3.2]{PKGV2024} it can be shown that the number of isomorphism classes of non-abelian cube-free imprimitive subgroups of order $4m$ is equal to $2|{\cal O}|$ where ${\cal O}$ is the set of all non central subgroups of $D(2,q)$ of order $m$ that are normal in $M(2,q)$. Let ${\cal O}_N$ be the set of all subgroups of $D(2,q)$ that are normal in $M(2,q)$ of order $m$. Clearly $|{\cal O}| = |{\cal O}_N|$ if ${\cal S} \neq {\cal P}$ and $|{\cal O}| = |{\cal O}_N| - 1$ if ${\cal S} = {\cal P}$. Let ${\cal O}_{I}$ be the complete and irredundant set of isomorphism classes of subgroups of $D(2,q)$ of order $m$ that are normal in $M(2,q)$. Then $|{\cal O}_N| = \sum_{H \in {\cal O}_I}|{\rm S}_H|$ where ${\rm {S}_H} = \{K \leq D(2,q)\mid aKa^{-1} = K$ and $K \cong H\}$. Let $H \in {\cal O}$ and let ${\cal I}(H) = \{i \mid$ the Sylow $p_i$-subgroup of $H$ is cyclic$\}$. Now if $r \in {\cal S}$, then there exists a unique subgroup $U$ of $D(2,q)$ isomorphic to ${\mathbb Z}_r \times {\mathbb Z}_r$. Therefore $|{\rm S}_H| = \prod_{i \in {\cal I}(H)}|{\rm S}_{P_i}|$ where $P_i$ is the Sylow $p_i$-subgroup of $H$ and ${\rm S}_{P_i} = \{K \leq D(2,q) \mid aP_ia^{-1} = P_i$ and $K \cong P_i\}$. If $i \in {\cal I}(H)$, then by \cite[Lemma 2.2]{PKGV2023}, $|{\rm S}_{P_i}| = 2$. So 
    \begin{align*}
        |{\cal O}_N| & = \sum_{H \in {\cal O}_I}|{\rm S}_{H}|\\
                     & = \sum_{d = 0}^{t} \sum_{H \in {\cal O}_d}|{\rm S}_{H}|\\
                     & = \sum_{d = 0}^{t}|{\cal O}_d|2^d \\
                     & = \sum_{d = 0}^{t} {\binom{t}{d}}2^d \\
                     & = 3^t.
    \end{align*}
    
   where ${\cal O}_d = \{H \in O \mid |{\cal I}(H)| = d\}$ and $t = |{\cal P}|$.
   
    \medskip
    Finally if $M$ is a non-abelian maximal cube-free primitive $p'$-subgroup of ${\rm GL}(2,q)$, then by \cite[Lemma 1.1]{PKGV2024}, $M$ is conjugate to a subgroup of $N(2,q)$ and $M = K \rtimes P$ where $K \leq S(2,q)$ and $P$ is a Sylow $2$-subbgroup of $M$. By maximality of $M$, we have that $K$ is the maximal cube-free odd order $p'$-subgroup of $S(2,q)$. Now if $4 \nmid |M|$, then we can choose an element of order $2$ from $S(2,q)$ and form a subgroup $L \cong M \times {\mathbb Z}_2$ which is a solvable cube-free primitive $p'$-subgroup and larger than $M$. Hence $4 \mid |M|$. Since $M = K \rtimes P$ where $K$ is a unique subgroup contained in $S(2,q)$, we will only have two choices for $M$ in this case. Also by  \cite[Theorem 1.3]{PKGV2024} any two isomorphic solvable cube-free primitive $p'$-subgroups are conjugate.
\end{proof}

Now we provide the list of subgroups of ${\rm GL}(2,q)$ that can be taken as the representatives of conjugacy classes of subgroups maximal amongst solvable cube-free imprimitive $p'$-subgroups of ${\rm GL}(2,q)$. 

\begin{rem}
    Assume the notation of Proposition \ref{No_of_cnjgcy_clsses_of_non_ablian_grps}. Let $r \in {\cal P}$ and let $\lambda_r \in {{\mathbb F}_q}^*$ be an element of order $r^2$. Let $a_r = {\rm diag}(\lambda_r,{\lambda_r}^{-1})$ and let $b_r = {\rm diag}({\lambda}_r,{\lambda_r})$. Then the subgroups $G_{1r} = \langle a_r\rangle$, $G_{2r} = \langle b_r\rangle$ and $G_{3r} = \langle {a_r}^r \rangle \times \langle {b_r}^r \rangle$ are normal in $M(2,q)$ and have order $r^2$. Let $u \in {\mathbb F}_q^*$ be the unique element of order $2$. Then $P_1 = \langle {\rm diag}(u,u) \rangle \times \langle a \rangle $ and ${P_2} = \langle {\rm diag}(1,u)a \rangle$ are subgroups of $M(2,q)$ of order $4$. Suppose that ${\cal S} \not = {\cal P}$. Let $l$ be the product of primes of ${\cal S}$ that are not in ${\cal P}$ and let $U$ be the unique subgroup of $D(2,q)$ isomorphic to ${\mathbb Z}_l \times {\mathbb Z}_l$. If ${\cal S} \not = {\cal P}$, then the set $\{P_{j} \ltimes ((\prod_{r \in {\cal P}}G_{ir}) \times U)\mid i \in \{1,2,3\},j \in \{1,2\} \}$ can be taken as the representatives of the conjugacy classes non-abelian maximal cube-free imprimitive $p'$-subgroups of ${\rm GL}(2,q)$. If ${\cal S} = {\cal P}$, then we remove subgroups $P_1 \times \prod_{r \in {\cal S}}\langle b_r\rangle$ and $P_2 \times \prod_{r \in \cal S}\langle b_r \rangle$ from $\{P_{j} \ltimes \prod_{r \in {\cal P}}G_{ir}\mid i \in \{1,2,3\},j \in \{1,2\} \}$ and obtain the representatives of the conjugacy classes non-abelian maximal cube-free imprimitive $p'$-subgroups of ${\rm GL}(2,q)$. 
\end{rem}

\begin{cor}\label{No_of_cnjgcy_clsses}
    The number of conjugacy classes of subgroups that are maximal amongst solvable cube-free $p'$-subgroups  of ${\rm GL}(2,p)$ is at most $\frac{1}{2}p^{2} - 1$.
    
\end{cor}

\begin{proof}
   For $p \in \{2,3,5\}$, it can be shown by standard arguments that the inequality holds. For $p \geq 7$ by Lemma \ref{No_of_cnjgcy_clsses_of_ablian_grps}, $(\rm ii)$  and by Proposition \ref{No_of_cnjgcy_clsses_of_non_ablian_grps} we have that $N \leq \frac{3}{4}(p-1)^{3/2} + \frac{3}{2}(p-1)^{1/2} + 3 < \frac{1}{2}p^2 - 1$. 
\end{proof}

\section{On the upper bound for cube-free groups}{\label{uppr_bound_of_cbe_free_grps}}

In this section we will find upper bounds for the number of cube-free groups of a given order up to isomorphism. Further we will see that the upper bound we provide for the number of cube-free solvable groups of a given order up to isomorphism is an improvement over the upper bound for the number of cube-free solvable groups of a given order up to isomorphism that we can obtain from \cite[Corollary 1.1]{GV1997}.

\medskip
We begin by finding the number of elementary abelian subgroups of a given order and the number of cyclic subgroups of order $p^2$ of a finite abelian $p$-group of exponent $p^2$.

\begin{lemm}\label{no_of_cube-free_subgrps}
    Let $Q = H \times N$ be an abelian $p$-group where $H \cong {\mathbb Z}_{p^2} \times \cdots \times {\mathbb Z}_{p^2}$ and $N = {\mathbb Z}_p \times \cdots \times {\mathbb Z}_p$ are subgroups of $Q$ of orders $p^{2u}$ and $p^{v}$ respectively. Let $s = u + v$. Then 

    \begin{enumerate}[$(\rm i)$]
        \item the number of elementary abelian subgroups of $Q$ of order $p^t$ is 
             $${\cal N} = \frac{(p^s-1)(p^{s-1} -1) \cdots (p^{s-(t-1)}-1)}{(p^t-1)(p^{t-1}-1)\cdots (p-1)};$$
        \item the number of cyclic subgroups of $Q$ of order $p^2$ is $p^{s-1}\frac{(p^u - 1)}{p-1}$.
    \end{enumerate}
    In particular, the number of subgroups of $Q$ of order $p^{\alpha}$ where $\alpha \in \{1,2\}$ is at most $2|Q|^{\alpha}/p^{\alpha}$.
\end{lemm}

\begin{proof}
    There exists a unique subgroup  $Q_p = \{x \in Q \mid x^p =1\}$ of $Q$ and every element of order $p$ is contained in $Q$. Since $Q_p \cong {\mathbb Z}_p \times \cdots \times {\mathbb Z}_p$ is a vector space of dimension $s$, the number of elementary abelian subgroups of $Q$ of order $p^t$ is equal to the number of subspaces of $Q_p$ of dimension $t$ and this number is equal to ${\cal N}$.  

    \medskip
    Since every element of order $p$ is contained in the subgroup $Q_p$, the number of elements of $Q$ of order $p^2$ is $|Q| - |Q_p|$. Therefore the number of cyclic subgroups of $Q$ of order $p^2$ is $(p^{s+u} - p^s)/p(p-1) = p^{s-1}\frac{(p^u -1)}{p-1}$. 
    \end{proof}

We now prove Theorem \ref{Enumeration_of_odd_cube_free_groups} for which we follow an approach similar to that of \cite[Theorem 12.9]{BNV2007}.

\medskip
\noindent
{\bf Proof of Theorem \ref{Enumeration_of_odd_cube_free_groups}}
Let $G$ be an odd order cube-free group of order ${p_1}^{\alpha_1} \ldots {p_k}^{\alpha_k}$ where the $p_i$'s are distinct primes and let $G_i = G/O_{p_i'}(G)$. Then $G$ is embedded in the group $G_1 \times \cdots \times G_k$ as a subdirect product. Note that $O_{p_i}(G_i) = F(G_i)$ the Fitting subgroup of $G_i$. Let $P_i$ be a Sylow $p_i$-subgroup of $G_i$. Since $G_i$ is a solvable cube-free group we have $C_{G_i}(F(G_i)) = F(G_i)$. Thus $P_i \leq C_{G_i}(O_{p_i}(G_i)) = C_{G_i}(F(G_i)) = F(G_i)$. Thus $P_i = O_{p_i}(G_i)$ and so $G_i \cong P_i \rtimes H_i$ with $H_i$ being a Hall $p_i'$-subgroup of $G_i$. Further since $H_i$ acts on $P_i$ via conjugation and the kernel of this action would be $P_i \cap H_i = \{1\}$, therefore $H_i \leq {\rm Aut}(P_i)$. 

Let $M_i \leq {\rm Aut}(P_i) $ be an odd order subgroup that is maximal among solvable cube-free $p_i'$-subgroups of ${\rm Aut}(P_i)$ such that $H_i \leq M_i$. Let $\hat{G_i} = P_i \rtimes M_i$ and let $\hat{G} = \hat{G_1} \times \cdots \times \hat{G_k} $. Then $G \leq \hat{G}$. The number of choices for $\hat{G}$ up to isomorphism is $ \leq \prod_{i=1}^{k}$ the number of choices for $\hat{G}_i$ up to isomorphism. Note that the isomorphism classes of $\hat{G}_{i}$ depends on the isomorphism class of $P_i$ and the conjugacy class of $M_i$ in ${\rm Aut}(P_i)$. Let ${\cal M}_i = [{\cal M}_{cf,odd,{p_i}'}(A_i)]$ be the set of conjugacy classes of odd order subgroups that are maximal amongst solvable cube-free ${p_i}'$-subgroups in $A_i$.  Clearly if $P_i$ is cyclic, then $|{\cal M}_i| = 1$. Therefore the number of choices for $\hat{G}_i$ up to isomorphism by Corollary \ref{No_of_odd_cnjgcy_clsses} is $$\prod_{i=1}^{k} \sum_{\substack{|P_i| = {p_i}^{\alpha_i} \\ P_i abelian}} |{\cal M}_i| \leq \prod_{i=1}^{k}2^{-1}{p_i}^{\frac{3}{4}{\alpha_i}}.$$

Now let $\hat{G} $ be fixed. Let $\{S_1, \ldots ,S_k \}$ be a Sylow system for $G$. So $S_i$ is a Sylow 
$p_i$-subgroup of $G$ and for all $i,j$ we have $S_iS_j = S_jS_i$. Thus $G = S_1S_2 \ldots S_k$. Further there exists $Q_1, \ldots , Q_k$, part of a Sylow system for $\hat{G}$ such that $S_i \leq Q_i.$ Using Lemma \ref{no_of_cube-free_subgrps} we get that the number of choices for $G$ as a subgroup of $\hat{G}$ (up to conjugacy) is at most 
\begin{equation*}
    |\{S_1, \ldots, S_k\} \mid S_i \leq Q_i \quad and \quad|S_i| = p_i^{\alpha_i} \}| \leq 2^{k}n^{-1}\prod_{i=1}^{k} |Q_i|^{\alpha_i}
\end{equation*}
where $Q_i$ is the Sylow $p_i$-subgroup of $\hat{G}$.

For each $i$, we have $Q_i = R_{i1} \times \cdots \times R_{ik}$ for some Sylow $p_i$-subgroups $R_{ij}$ of $\hat{G_j}$. Note that $|R_{ii}| = p_i^{\alpha_i}$. Let 
\begin{align*}
    X_j = \prod^{k}_{i=1, i\neq j}  R_{ij}. 
\end{align*}
Since the $Q_i$ form part of a Sylow system for $\hat{G}$, we have that $Q_iQ_j = Q_jQ_i$. So $X_j$ is a a solvable, cube-free  $p_i'$-subgroup of $\hat{G_j} = P_jM_j$. So $X_j$ is isomorphic to a subgroup of $M_j$.

Let $\mu = \mu(n)$. Then

\begin{align*}
    \prod_{i=1}^{k} |Q_i|^{\alpha_i} & = \prod_{i=1}^{k} \prod_{j=1}^{k}|R_{ij}|^{\alpha_i} \\
    & \leq   \prod_{t=1}^{k}(|R_{tt}|^{\alpha_t})(\prod_{i \neq j} |R_{ij}|)^{\mu} \\
    & = \prod_{i=1}^{k} p_i^{\alpha_i^2} \prod_{j=1}^{k} |X_j|^{\mu}.
\end{align*}
Using \cite[Lemma 1.1]{PKGV2024} we can show that $|X_j| \leq |M_j| \leq p_j^{\alpha_j}/2$. Thus 

\begin{align*}
    \prod_{i=1}^{k} |Q_i|^{\alpha_i} 
    & \leq \prod_{i=1}^{k}p_i^{\alpha_i^2} \prod_{j=1}^{k} 2^{-\mu} p_j^{\alpha_j \mu} \\
    & \leq 2^{-k\mu} n^{\mu} \prod_{i=1}^{k} p_i^{\alpha_i^2}     
\end{align*}

Putting together all the estimates, we get

\begin{align*}
    f_{cf,odd}(n) & \leq 2^{-k(\mu -1)} n^{\mu -1} \prod_{i=1}^{k} p_i^{\alpha_i^2} \prod_{i=1}^{k} 2^{-1}{p_i^{\frac{3}{4}{\alpha_i}}}\\
    & \leq  2^{-k\mu} n^{\mu -1/4} \prod_{i=1}^{k}p_i^{\alpha_i^2}\\
     & \leq 2^{-k\mu} n^{\mu -1/4} \prod_{i=1}^{k}p_i^{\alpha_i \mu} \\
     & =  2^{-\omega(n)\mu(n)}n^{2\mu(n) - 1/4}.
\end{align*}

\medskip
\noindent
{\bf Proof of Theorem \ref{Enumeration_of_cube_free_groups} } 
Let $G$ be a solvable cube-free group of order ${p_0}^{\alpha_0}{p_1}^{\alpha_1} \ldots {p_k}^{\alpha_k}$ where $p_0 = 2$ and the $p_i$'s for $i \geq 1$ are distinct odd primes. The proof is exactly the same as in Theorem \ref{Enumeration_of_odd_cube_free_groups} except that  the $M_i \leq {\rm Aut}(P_i)$ is now maximal amongst solvable cube-free $p_i'$-subgroup of ${\rm Aut}(P_i)$. We use Corollary \ref{No_of_cnjgcy_clsses} instead of Corollary \ref{No_of_odd_cnjgcy_clsses} to get 

\begin{align*}
    f_{cf,sol}(n) & \leq   n^{\mu} \prod_{i=0}^{k} p_i^{\alpha_i^2}\\
     & \leq  n^{\mu} \prod_{i=0}^{k}p_i^{\alpha_i \mu} \\
     & = n^{2\mu(n)}.    
\end{align*}

\begin{rem}
     Let $G$ be a cube-free non-solvable group of order $n$. Then by \cite[Theorem 3.10]{QL2011}, $G = PSL(2,r) \times L $, where r is some suitable prime and $|L|$ is of odd order. Using this result we can estimate  that the number of non-solvable non-isomorphic cube-free groups of order $n$ is roughly of $O(n^5)$.
\end{rem}

    
 


    

\section{On the lower bound for cube-free groups}{\label{lowr_bound_of_cbe_free_grps}}

 We shall prove Theorem \ref{isomorphism_1} and Theorem \ref{lowr_bound} in this section. Let $p$ be a prime. Let $Q$ be a cube-free abelian group of odd order and let $p \nmid |Q|$. As stated earlier Theorem \ref{isomorphism_1}  will provide us the number of split extensions of ${\mathbb Z}_p \times {\mathbb Z}_p$ by $Q$  and the number of split extensions of ${\mathbb Z}_{p^{\alpha}}$ by $Q$ where $\alpha$ is a positive integer. The lower bound for the number of cube-free groups of a given order up to isomorphism is presented in Theorem \ref{lowr_bound} and will be given in terms of the formulae we obtain in Theorem \ref{isomorphism_1}.

\medskip
Recall that if  $N$ and $Q$ are finite abelian groups with $\gcd(|N|,|Q|)=1$, ${\cal X}_{N,Q} = {\rm Aut}(N) \times {\rm Aut}(Q)$ and ${\Gamma}_{N,Q} = {\rm Hom}(Q,{\rm Aut(N)})$, then ${\cal X}_{N,Q}$ acts on ${\Gamma}_{N,Q}$ by the equation (\ref{eq: action}).

\medskip
 The result below is a generalisation of \cite[Proposition 18.2]{BNV2007} and the proof is quite similar. We use the notation established above.

 \begin{prop}{\label{crrspndnc_thm}}
There is a one-to-one correspondence between the isomorphism classes of groups $N \rtimes Q$ and the set of orbits of ${\cal X}_{N,Q}$ on $\Gamma_{N,Q}$.

 \end{prop}

 \begin{proof}
 We need to prove that $N \rtimes_{\theta} Q \cong N \rtimes_{\theta'} Q$ if and only if $\theta$ and $\theta'$ lie in the same orbit of ${\cal X}_{N,Q}$.

 \medskip
 Let $(\kappa, \lambda) \in {\cal X}_{N,Q}$ be such that $(\kappa,\lambda)\theta = \theta'$, so $\theta$ and $\theta'$ satisfy (\ref{eq: action}). Define the map $\phi: N \rtimes_{\theta} Q \longrightarrow N \rtimes_{\theta'} Q$ by $\phi(a,x) = (\kappa(a),\lambda(x))$ for all $a \in N$ and $x \in Q$. It can be checked that $\phi$ is an isomorphism.

 \medskip
 Conversely, suppose that  $\phi: N \rtimes_{\theta} Q \longrightarrow N \rtimes_{\theta'} Q$ is an isomorphism. Now $\phi(1 \times Q)$ is a Hall $|N|'$-subgroup of $N \rtimes_{\theta'} Q$, there exists an inner automorphism $\gamma$ of $N \rtimes_{\theta'} Q$ such that $\gamma(\phi(1 \times Q)) = 1 \times Q$. We may therefore assume $\phi(1 \times Q) = 1 \times Q$. Clearly $\phi(N \times 1) = N \times 1$. Define $\kappa: N \longrightarrow N$ and $\lambda: Q \longrightarrow Q$ as 
 \begin{center}
      $(\kappa(a),1) = \phi((a,1))$ for all $a \in N$ and
 \end{center}

 \begin{center}
     $(1,\lambda(x)) = \phi((1,x))$ for all $x \in Q$.
 \end{center}
It is easy to check that $\kappa \in {\rm Aut}(N)$ and $\lambda \in {\rm Aut}(Q)$. Thus for all $a \in N$ and $x \in Q$

\begin{center}
    $\phi((1,x))\phi((a,1)) = (1,\lambda(x))(\kappa(a),1) = (\theta'_{\lambda(x)}(\kappa(a)),\lambda(x))$ and 
\end{center}

\begin{center}
    $\phi((1,x)(a,1)) = \phi(\theta_x(a),x) = (\kappa(\theta_x(a), \lambda(x)))$.
\end{center}
Thus $\theta'_{\lambda(x)}(\kappa(a)) = \kappa\theta_{x}(a)$ for all $a \in N$ and $x \in Q$ which is same as (\ref{eq: action}).
     
 \end{proof}

\begin{prop}{\label{no_of_iso_classes}}
    Let $N$ be a finite abelian group and let $Q$ be a finite abelian group whose order is co-prime to $|N|$. Let the Sylow subgroups of $Q$ be cyclic or elementary abelian. Then $N \rtimes_{\theta_1} Q$ is isomorphic to $N \rtimes_{\theta_2} Q$ if and only if $\theta_1(Q)$ is conjugate to $\theta_2(Q)$ in ${\rm Aut}(N)$.
\end{prop}

\begin{proof}
    If $N \rtimes_{\theta_1} Q$ is isomorphic to $N \rtimes_{\theta_2} Q$, then by Proposition \ref{crrspndnc_thm}, $\theta_1(Q)$ and $\theta_2(Q)$ are conjugate in ${\rm Aut}(N)$.

    \medskip
    Suppose $\theta_1(Q)$ and $\theta_2(Q)$ are conjugate in ${\rm Aut}(N)$, then there exsits $\gamma \in {\rm Aut}(N)$ such that $\theta_2(Q) = \gamma\theta_1(Q)\gamma^{-1}$. Let $q_i$ be a prime dividing $|Q|$ and let $Q_i$ be the Sylow $q_i$-subgroup of $Q$. Then since $Q$ is abelian we have $\theta_2(Q_i) = \gamma\theta_1(Q_i)\gamma^{-1}$. Since $Q$ is abelian it is enough to show that for each prime divisor $q_i$ of $|Q|$ there exists a $\lambda_i \in {\rm Aut}(Q_i)$ such that $\theta_2(a) = \gamma\theta_1(\lambda_i(a))\gamma^{-1}$ for all $a \in Q_i$.
     
     First suppose $Q_i \leq ker(\theta_1)$, then $Q_i \leq ker(\theta_2)$ and we can take $\lambda_i$ as an identity automorphism. If $Q_i$ is cyclic or an elementary abelian $q_i$-subgroup of $Q$, then by \cite[Corollary 3.10]{SC2024} or by \cite[Proposition 4.1]{SC2024} we get the desired $\lambda_i$.
    \end{proof}

Now we shall prove Theorem \ref{isomorphism_1}. We use the notation established before the statement of Theorem \ref{isomorphism_1}.

\medskip
{\bf Proof of Theorem 1.3} Let $N = {\mathbb Z}_p \times {\mathbb Z}_p$ and let $Q = {\mathbb Z}_{cd} \times {{\mathbb Z}_d}$. Let ${\cal O}$ be the set of all images of $Q$ up to isomorphism in ${\rm GL}(2,p)$. Then by Proposition \ref{no_of_iso_classes} we have

\begin{align*}
    f_{\cal S}(p^2cd^2)
      = \sum_{\substack{H \in {\cal O} \\ H red}}N_{red}(H(t)) + \sum_{\substack{H \in {\cal O}\\ H irred}}N_{irr}(N(H(t))
\end{align*}

Now we consider the case when $H \in {\cal O}$ and $H$ is a reducible $p'$-subgroup of ${\rm GL}(2,p)$. Then by \cite[Lemma 1.1]{PKGV2024} $H \cong {\mathbb Z}_u \times {\mathbb Z}_v$ where $u \mid p-1$ and $v \mid u$. Note that $u$ is cube-free and $v$ is square-free. Suppose $\theta(Q) \cong {\mathbb Z}_u \times {\mathbb Z}_v$ where $u \mid p-1$ and $v \mid u$ for some $\theta \in \Gamma_{N,Q}$, then since $u$ is the exponent of $\theta(Q)$ we have $u \mid \gcd(cd,p-1)$. Also by \cite[Theorem 10.57]{JJR1995} there is a subgroup $\hat{K} \leq Q$ such that $\theta(Q) \cong \hat{K}$. If $r$ is a prime dividing $v$, then the Sylow $r$-subgroup of $\hat{K}$ is isomorphic to ${\mathbb Z}_r \times {\mathbb Z}_r$ therefore, we must have $v \mid d$. Thus $v \mid \gcd(d,p-1)$. Further it is not difficult to see that there exists a $\phi  \in {\Gamma}_{N,Q}$ such that $\phi(Q) \cong {\mathbb Z}_l \times {\mathbb Z}_m$ where $l = {\rm gcd}(cd,p-1)$ and $m = {\rm gcd}(d,p-1)$. Fix $K \leq Q$ such that $K \cong {\mathbb Z}_l \times {\mathbb Z}_m$. Since $Q$ is a cube-free abelian group any two subgroups of $Q$ of same order are isomorphic. Therefore an isomorphic copy of a reducible image of $Q$ is contained in $K$. Therefore

\begin{equation*}
    f_{{\cal S}}(p^2cd^2)  = \sum_{\substack{t \mid lm \\ H \leq K}}N_{red}(H(t)) + \sum_{H \in {\cal O}}N_{irr}(N(H(t))
\end{equation*}

Since any two subgroups of $Q$ of same order are isomorphic we can simply write 

\begin{equation*}\label{eq:no_of_iso_clsses}
    f_{\cal S}(p^2cd^2)  = \sum_{{t \mid lm}}N_{red}(H(t)) + \sum_{H \in {\cal O}}N_{irr}(N(H(t))\tag{$**$}
\end{equation*}

\medskip
Now if $\theta \in {\Gamma}_{N,Q}$ is irreducible then $\theta(Q)$ is cyclic and by \cite[Theorem 2.3.3]{S1992} the number of conjugacy classes of irreducible images of $Q$ is equal to the number of irreducible images of $Q$. Further by \cite[Theorem 2.3.2]{S1992} there exists an irreducible representation of $Q$ of degree $2$ if and only if $s \neq 1$ where $s = \gcd(cd,p+1)$. Let $s \neq 1$. Then we can show easily that there exists an irreducible representation $\theta \in \Gamma_{N,Q}$ with $\theta(Q) \cong {\mathbb Z}_t$ if and only if $t \mid \gcd(cd,p^2-1)$ and $t \nmid p-1$. Since $c$ and $d$ are odd positive integers we have $\gcd(s,l) = 1$ where $l = \gcd(cd,p-1)$, therefore $\gcd(p^2-1,cd) = sl$. It is not difficult to see that the number of divisors of $\gcd(p^2-1,cd)$ that do not divide $p-1$ is equal to $\tau(sl)-\tau(l) = (\tau(s)-1)\tau(l)$. By the equation (\ref{eq:no_of_iso_clsses}) we have

\begin{align*}
     f_{{\cal S}}(p^2cd^2) = \sum_{t \mid lm}N_{red}(H(t)) + \Omega(s,l).
\end{align*}

\medskip
 Let $N = {\mathbb Z}_{p^{\alpha}}$ and let $Q = {\mathbb Z}_c \times {{\mathbb Z}_d}^2$. Since ${\rm Aut}({\mathbb Z}_{p^{\alpha}}) \cong {\mathbb Z}_{p^{\alpha} - p^{\alpha -1}}$ is cyclic, for each $\theta \in \Gamma_{N,Q}$ the image $\theta(Q)$ is cyclic and $|\theta(Q)| \mid \gcd(cd,p-1)$. Further it can easily be seen that there exists a $\phi \in {\Gamma_{N,Q}}$ such that $\phi(Q) \cong {\mathbb Z}_l$. Thus by Proposition \ref{no_of_iso_classes},  $f_{{\cal S}}(p^{\alpha}cd^2) = \tau(l)$.

\medskip
{\bf Proof of Corollary 1.4} Let $N = {\mathbb Z}_{ab} \times {\mathbb Z}_b$. Then since $N = P_1 \times \cdots \times P_k$ we have ${\rm Aut}(N) = {\rm Aut}(P_1) \times \cdots \times {\rm Aut}(P_k)$. Let ${\cal A} = {\rm Aut}(N)$ and let ${\cal P}_i = {\rm Aut}(P_i)$. Let ${\cal O}$ be set of all conjugacy classes of the subgroups of ${\cal A}$ and let ${\cal O}_i$ be the set of all conjugacy classes of the subgroups of ${\cal P}_i$. Then we define a map $\psi: {\cal O}_1 \times \cdots \times {\cal O}_k  \longrightarrow {\cal O}$ as $([H_1]_{{\cal P}_1}, \ldots , [H_k]_{{\cal P}_k}) \longmapsto [H]_{\cal A}$ where $H = H_1 \times \cdots \times H_k$. It is easy to check that $\psi$ is injective and the result follows from Proposition \ref{no_of_iso_classes}.

\section{Acknowledgements}

\noindent Prashun Kumar would like to acknowledge the UGC-SRF grant ({\emph{identification number}}: 201610088501) which is enabling his doctoral work.

\end{document}